\documentclass[10pt,reqno]{amsart}
\usepackage{latexsym,amsmath,amssymb,amscd}
\usepackage[all]{xy}

\begin{document}

\def\theequation{
\arabic{equation}}

\newcommand{\bfi}{\bfseries\itshape}

\newtheorem{theorem}{Theorem}
\newtheorem{acknowledgment}[theorem]{Acknowledgment}
\newtheorem{corollary}[theorem]{Corollary}
\newtheorem{definition}[theorem]{Definition}
\newtheorem{example}[theorem]{Example}
\newtheorem{lemma}[theorem]{Lemma}
\newtheorem{notation}[theorem]{Notation}
\newtheorem{proposition}[theorem]{Proposition}
\newtheorem{remark}[theorem]{Remark}
\newtheorem{setting}[theorem]{Setting}
\newtheorem{hypothesis}[theorem]{Hypothesis}

\numberwithin{theorem}{section}
\numberwithin{equation}{section}

\renewcommand{\1}{{\bf 1}}
\newcommand{\Ad}{{\rm Ad}}
\newcommand{\Alg}{{\rm Alg}\,}
\newcommand{\Aut}{{\rm Aut}\,}
\newcommand{\ad}{{\rm ad}}
\newcommand{\Borel}{{\rm Borel}}
\newcommand{\botimes}{\bar{\otimes}}
\newcommand{\Ci}{{\mathcal C}^\infty}
\newcommand{\Cint}{{\mathcal C}^\infty_{\rm int}}
\newcommand{\Cpol}{{\mathcal C}^\infty_{\rm pol}}
\newcommand{\Der}{{\rm Der}\,}
\newcommand{\Diff}{{\rm Diff}\,}
\newcommand{\de}{{\rm d}}
\newcommand{\ee}{{\rm e}}
\newcommand{\End}{{\rm End}\,}
\newcommand{\ev}{{\rm ev}}
\newcommand{\hotimes}{\widehat{\otimes}}
\newcommand{\id}{{\rm id}}
\newcommand{\ie}{{\rm i}}
\newcommand{\iotaR}{\iota^{\rm R}}
\newcommand{\GL}{{\rm GL}}
\newcommand{\gl}{{{\mathfrak g}{\mathfrak l}}}
\newcommand{\Hom}{{\rm Hom}\,}
\newcommand{\Img}{{\rm Im}\,}
\newcommand{\Ind}{{\rm Ind}}
\newcommand{\Ker}{{\rm Ker}\,}
\newcommand{\Lie}{\text{\bf L}}
\newcommand{\Mt}{{{\mathcal M}_{\text t}}}
\newcommand{\m}{\text{\bf m}}
\newcommand{\pr}{{\rm pr}}
\newcommand{\Ran}{{\rm Ran}\,}
\renewcommand{\Re}{{\rm Re}\,}
\newcommand{\so}{\text{so}}
\newcommand{\spa}{{\rm span}\,}
\newcommand{\Tr}{{\rm Tr}\,}
\newcommand{\tw}{\ast_{\rm tw}}
\newcommand{\Op}{{\rm Op}}
\newcommand{\U}{{\rm U}}
\newcommand{\UCb}{{{\mathcal U}{\mathcal C}_b}}
\newcommand{\weak}{\text{weak}}

\newcommand{\CC}{{\mathbb C}}
\newcommand{\RR}{{\mathbb R}}
\newcommand{\TT}{{\mathbb T}}
\newcommand{\NN}{{\mathbb N}}

\newcommand{\Ac}{{\mathcal A}}
\newcommand{\Bc}{{\mathcal B}}
\newcommand{\Cc}{{\mathcal C}}
\newcommand{\Dc}{{\mathcal D}}
\newcommand{\Ec}{{\mathcal E}}
\newcommand{\Fc}{{\mathcal F}}
\newcommand{\Hc}{{\mathcal H}}
\newcommand{\Jc}{{\mathcal J}}
\newcommand{\Lc}{{\mathcal L}}
\renewcommand{\Mc}{{\mathcal M}}
\newcommand{\Nc}{{\mathcal N}}
\newcommand{\Oc}{{\mathcal O}}
\newcommand{\Pc}{{\mathcal P}}
\newcommand{\Qc}{{\mathcal Q}}
\newcommand{\Sc}{{\mathcal S}}
\newcommand{\Tc}{{\mathcal T}}
\newcommand{\Vc}{{\mathcal V}}
\newcommand{\Uc}{{\mathcal U}}
\newcommand{\Xc}{{\mathcal X}}
\newcommand{\Yc}{{\mathcal Y}}
\newcommand{\Wig}{{\mathcal W}}

\newcommand{\Bg}{{\mathfrak B}}
\newcommand{\Fg}{{\mathfrak F}}
\newcommand{\Gg}{{\mathfrak G}}
\newcommand{\Ig}{{\mathfrak I}}
\newcommand{\Jg}{{\mathfrak J}}
\newcommand{\Lg}{{\mathfrak L}}
\newcommand{\Pg}{{\mathfrak P}}
\newcommand{\Sg}{{\mathfrak S}}
\newcommand{\Xg}{{\mathfrak X}}
\newcommand{\Yg}{{\mathfrak Y}}
\newcommand{\Zg}{{\mathfrak Z}}

\newcommand{\ag}{{\mathfrak a}}
\newcommand{\bg}{{\mathfrak b}}
\newcommand{\dg}{{\mathfrak d}}
\renewcommand{\gg}{{\mathfrak g}}
\newcommand{\hg}{{\mathfrak h}}
\newcommand{\kg}{{\mathfrak k}}
\newcommand{\mg}{{\mathfrak m}}
\newcommand{\n}{{\mathfrak n}}
\newcommand{\og}{{\mathfrak o}}
\newcommand{\pg}{{\mathfrak p}}
\newcommand{\sg}{{\mathfrak s}}
\newcommand{\tg}{{\mathfrak t}}
\newcommand{\ug}{{\mathfrak u}}
\newcommand{\zg}{{\mathfrak z}}

\newcommand{\ZZ}{\mathbb Z}
\newcommand{\BB}{\mathbb B}
\newcommand{\HH}{\mathbb H}

\newcommand{\ep}{\varepsilon}

\newcommand{\hake}[1]{\langle #1 \rangle }

\newcommand{\scalar}[2]{(#1 \mid#2) }
\newcommand{\dual}[2]{\langle #1, #2\rangle}

\newcommand{\norm}[1]{\Vert #1 \Vert }

\makeatletter
\title[Boundedness for Weyl-Pedersen calculus]{Boundedness for Weyl-Pedersen calculus\\ on flat coadjoint orbits}
\author{Ingrid Belti\c t\u a 
 and Daniel Belti\c t\u a
}
\address{Institute of Mathematics "Simion Stoilow" 
of the Romanian Academy, Research Unit~1,  
P.O. Box 1-764, Bucharest, Romania}
\email{Ingrid.Beltita@imar.ro, ingrid.beltita@gmail.com}
\email{Daniel.Beltita@imar.ro, beltita@gmail.com}
\date{\today}
\makeatother

\begin{abstract}
We describe boundedness and compactness properties for the operators obtained by 
the Weyl-Pedersen calculus
in the case of the irreducible unitary representations of nilpotent Lie groups 
that are associated with flat coadjoint orbits. 
We use spaces of smooth symbols satisfying appropriate growth conditions 
expressed in terms of invariant differential operators on the coadjoint orbit 
under consideration. 
Our method also provides conditions for these operators to 
belong to one of the Schatten ideals of compact operators. 
In the special case of the Schr\"odinger representation of the Heisenberg group we recover 
some classical properties of the pseudo-differential Weyl calculus, 
as the Calder\'on-Vaillancourt theorem, and the Beals characterization in terms of commutators.
\end{abstract}

\maketitle

\section{Introduction}\label{introduction}

We aim for describing boundedness properties for the operators obtained by 
the Weyl-Pedersen calculus (\cite{Pe94}) 
in the case of the irreducible unitary representations of nilpotent Lie groups 
that are associated with flat coadjoint orbits. 
To this end we use spaces of smooth symbols satisfying appropriate growth conditions 
expressed in terms of invariant differential operators on the coadjoint orbit 
under consideration. 
In turn, these spaces of symbols are invariant under the coadjoint action. 
Our method is inspired by \cite{Ka76} and also provides conditions for the aforementioned operators to be compact or to 
belong to one of the Schatten ideals of compact operators. 
In the special case of the Schr\"odinger representation of the Heisenberg group the invariant differential operators are precisely the linear partial differential operators with constant coefficients. 
Thus we recover some classical properties of the pseudo-differential Weyl calculus, 
which go back to \cite{CV72}, \cite{Be77}, and \cite{Ro84}. 

The problem of finding sufficient conditions 
for the boundedness of pseudo-differential operators associated with unitary representations
 of various types of nilpotent Lie groups received much attention: for smooth symbols of convolution operators 
(\cite{Ho84}), on 2-step nilpotent groups (\cite{Mi82}), on 3-step nilpotent groups (\cite{Ra85}), 
on graded groups (\cite{Me83}, 
\cite{Gl07}), 
and for non-smooth symbols 
(\cite{BB10c}, \cite{BB12}). 
There is also a vast literature on the boundedness of singular integral operators 
on nilpotent Lie groups 
(see for instance \cite{Mu83}, \cite{Mu84}). 
In contrast to these investigations, we will work below with groups whose generic coadjoint orbits are flat.  
These groups are not necessarily graded (as the examples of \cite{Bu06} show), 
and on the other hand every nilpotent Lie group embeds into a group of this type 
as a closed subgroup (see \cite[Th. 2.1]{Co83}), 
hence there exist groups of arbitrarily high nilpotency step to which our results apply. 
Moreover, the growth conditions that we use are different from the ones already used in the literature, 
inasmuch as we use invariant differential operators on coadjoint orbits. 
We also give a Beals-type characterisation of this space of symbols. 
 
\subsection*{Statement of the main results}
To describe the contents of our paper in more detail, 
let $G$ be a connected, simply connected, nilpotent Lie group 
with the Lie algebra $\gg$, whose center is denoted by~$\zg$. 
Let $\pi\colon G\to\BB(\Hc)$ be a unitary irreducible representation 
associated with the coadjoint orbit $\Oc\subseteq\gg^*$. 
We define $\Diff(\Oc)$ as the space of all linear differential operators $D$ on $\Oc$ 
which are \emph{invariant} to the coadjoint action, in the sense that   
 $$
(\forall g\in G)(\forall a\in C^\infty(\Oc))\quad  D(a\circ\Ad_G^*(g)\vert_{\Oc})=(Da)\circ\Ad_G^*(g)\vert_{\Oc}.$$
We will henceforth assume that $\Oc$ is a \emph{generic flat coadjoint orbit}. 
This is equivalent to the condition $\dim\Oc=\dim\gg-\dim\zg$, and it is also equivalent to the fact that the representation $\pi$ is square integrable modulo the center of $G$. (See \cite{CG90}.) 
Then the Weyl-Pedersen calculus $\Op\colon\Sc'(\Oc)\to\Lc(\Hc_\infty,\Hc_{-\infty})$ 
is a linear topological isomorphism which is uniquely determined by the condition 
that for every $b\in\Sc(\gg)$ we have  
$$\Op(\check{b}\vert_{\Oc})=\int\limits_{\gg}\pi(\exp_G X) b(X)\de X,$$ 
where $\check{b}(\xi)=\int_{\gg}\ee^{\ie\langle\xi,Y\rangle}b(Y)\de Y$ for all $\xi\in\gg^*$ 
and $\langle\cdot,\cdot\rangle\colon\gg^*\times\gg\to\RR$ stands for the duality pairing 
(see \cite[Th. 4.2.1]{Pe94}, \cite{Ma07}). 
We have also used above the notation $\Hc_\infty$ for the nuclear Fr\'echet space of smooth vectors of $\pi$,  
$\Hc_{-\infty}$ for the space of antilinear continuous functionals thereon, 
$\Lc(\Hc_\infty,\Hc_{-\infty})$ for the space of continuous linear operators between the above space 
(these operators are thought of as possibly unbounded linear operators in $\Hc$), 
and $\Sc(\bullet)$ and $\Sc'(\bullet)$ for the spaces of Schwartz functions and tempered distributions, respectively. 

Let us consider the space of symbols
\begin{equation}\label{symbols}
\Cc^\infty_b(\Oc) =\{a\in C^\infty(\Oc)\mid  Da\in L^\infty(\Oc)\;\text{for all} \; D\in\Diff(\Oc)\},
\end{equation} 
with the Fr\'echet topology given by the seminorms $\{a\mapsto\Vert Da\Vert_{L^\infty(\Oc)}\}_{D\in\Diff(\Oc)}$.
The main results of the present paper can be then summarized in the next theorem.

\begin{theorem}\label{main_introd}
Let $G$ be a connected, simply connected, nilpotent Lie group whose generic coadjoint orbits are flat. 
Let $\Oc$ be such an orbit with a corresponding unitary irreducible representation $\pi\colon G\to\BB(\Hc)$.
Then the following assertions are equivalent for $a\in C^\infty(\Oc)$:
\begin{enumerate}
\item  $a\in \Cc^\infty_b(\Oc)$,
\item For every $D\in\Diff(\Oc)$ we have $\Op(Da)\in\BB(\Hc)$. 
\end{enumerate}
Moreover the Weyl-Pedersen calculus defines a continuous linear map 
$$\Op\colon \Cc^{\infty}_b(\Oc)\to\BB(\Hc),$$
and the Fr\'echet topology of $C_b^\infty(\Oc)$ is equivalent to that defined by the family of seminorms 
$\{a\mapsto\norm{\Op(Da)}\}_{D\in\Diff(\Oc)}$. 
\end{theorem}

Let $\Cc^\infty_\infty(\Oc)$ be the space of all $a\in \Cc^\infty(\Oc)$ such that 
the function $Da$ vanishes at infinity on $\Oc$, for every $D\in \Diff(\Oc)$. 
Then $\Cc^\infty_\infty(\Oc)$ is the closure of $\Sc(\Oc)$ in $\Cc_b^\infty(\Oc)$. 
Since on the other hand the set $\Op(\Sc(\Oc))$ is dense in the space $\Sg_\infty(\Hc)$ 
of compact operators on $\Hc$ (see for instance \cite[Cor. 3.3 and Thm. 4.1]{BB10b}), it then follows by
 the above theorem that 
 $$ \Cc_{\infty}^{\infty}(\Oc) = \{ a \in  \Cc^\infty(\Oc)\mid  \Op(Da)\in \Sg_\infty(\Hc) \;\text{for all} \; D\in\Diff(\Oc)\}.$$
 See Theorem~\ref{main_Schatten} below for a similar statement involving Schatten ideals of compact operators.

If $\pi$ is the Schr\"odinger representation of the $(2n+1)$-dimensional 
Heisenberg group, then Theorem~\ref{main_introd} 
gives the characterization of the symbols of type $S^{0}_{0,0}$ 
for the pseudo-differential Weyl calculus 
$\Op\colon\Sc'(\RR^{2n})\to\Lc(\Sc(\RR^n),\Sc'(\RR^n))$.  
Namely,  for any symbol $a\in C^\infty(\RR^{2n})$ we have  
$$(\forall\alpha\in\NN^{2n})\quad\partial^\alpha a\in L^\infty(\RR^{2n})\iff 
(\forall\alpha\in\NN^{2n})\quad \Op(\partial^\alpha a)\in\BB(L^2(\RR^n)),$$  
where  $\partial^\alpha $ stand as usually for the partial derivatives; 
see for instance \cite{Bo97}. 

\subsection*{Outline of the paper} 
In Section~\ref{section2} we give some definitions and restate the main result 
in a more explicit form (Theorem~\ref{mainthm1} and Theorem~\ref{mainthm2}). 
In Section~\ref{section-conv} we provide some auxiliary results on convolutions for certain transformation groups and the Weyl-Pedersen calculus. 
The main result is proved in Section~\ref{section-proofs}. 
In Section~\ref{computations} we give a method of computing the invariant differential operators on a flat orbit and illustrate it by two examples that correspond to items N5N3 and N5N6 in~\cite{Pe88}. 
And finally, we apply this method in Section~\ref{final} for convolution operators on the Heisenberg group 
and obtain a characterisation of the $L^2$-bounded convolution operators.

We refer to \cite{CG90} for background information on representation theory of nilpotent Lie groups.

\section{Smooth functions on the coadjoint orbit}\label{section2}

In this section we show that the space $\Cc^\infty_b(\Oc)$ can be given in terms of some infinitesimal vector fields of an action of $G$ on $\Oc$, and we restate Theorem~\ref{main_introd} accordingly.

In the rest of the paper we shall identify $G$ with $\gg$, by using the exponential map, so that 
$G= (\gg, \cdot_G)$, where $\cdot_G$ is the Baker-Campbell-Hausdorff multiplication. 
We shall however keep notation as $G$,  $\Ad_G$ each time when it is important to point out that operations are considered on the group.

We may assume that the center $\zg$ of $\gg$ is one dimensional. 
Let $X_0, X_1, \dots, X_d$ be a Jordan-H\"older basis such that  $X_0$ generates $\zg$,
and denote by $\xi_0, \xi_1, \dots, \xi_d\in \gg^\ast$ the dual basis. 
Recall that the coadjoint orbit $\Oc$ is assumed to be flat and of dimension $d$, hence
$\Oc=\xi_0+\zg^\perp= \{\xi\in \gg^\ast\mid \scalar{\xi}{X_0}=1\}$.

Denote by  $\gg_0=\mathrm{span}\{X_1, \dots, X_d\}$.
Since $\gg= \zg \dot{+} \gg_{e}$, we have the linear isomorphism $ \gg_0\simeq \gg/\zg$, thus $\gg_0$ has the structure of a nilpotent Lie algebra.
We denote by $\cdot$ the Baker-Campbell-Hausdorff multiplication on $\gg_0$, by $G_0= (\gg_0, \cdot)$ the corresponding group, and by $\Ad$ the adjoint representation associated to $G_0$. 
Note that 
$$ 
X\cdot Y- X\cdot_G Y,\, \Ad(X) Y -\Ad_G(X) Y   \in \zg,   \; \; 
\Ad_G^{\ast}( X\cdot Y)= \Ad_G^\ast (X\cdot_G Y), 
$$
for every  $X, Y\in \gg_0$.

\begin{definition}\label{axi0} 
\normalfont
 We will need a global chart and a parameterization of $\Oc$, which can be constructed as follows: 
\begin{enumerate}
\item The simplest global chart is 
$$P\colon\Oc\to\gg_0^*,\quad P(\xi)=\xi\vert_{\gg_0}.$$
For every $\eta\in\gg_0^*$, the functional $\xi:=P^{-1}(\eta)\in\gg^*$ is uniquely determined 
by the conditions $\xi\vert_{\gg_0}=\eta$ and $\langle\xi,X_0\rangle=1$. 
\item The global parameterization defined by  the coadjoint action of $G$, 
$$A_{\xi_0}\colon\gg_0\to\Oc,\quad A_{\xi_0}(X)=\Ad_G^*(X)\xi_0=\xi_0\circ\ee^{-\ad_{\gg}X}.$$
\end{enumerate}
\end{definition}

Recall that $A_{\xi_0}$ is a polynomial diffeomorphism (see \cite{Pe89}). 
It induces the isomorphism
 $$ A_{\xi_0}^\ast \colon \Sc'(\Oc) \xrightarrow{\sim} \Sc'(\gg_0), \quad (A_{\xi_0}^\ast a)(X) = a (A_{\xi_0}( X)).$$

\begin{definition}\label{actions}
\normalfont
 Consider the following actions of $\gg_0$ in spaces of distributions on $\gg_0$ and  $\Oc$, respectively:
 
 \begin{enumerate}
  \item \label{actions_1} $\lambda, \rho\colon \gg_0\to \End(\Sc'(\gg_0))$ are induced by the left, respectively right, 
  regular representations of $\gg_0$ by 
  $$ (\lambda(X) b)(Y) = b((-X)\cdot Y), \quad (\rho(X)b)(Y) = b(Y\cdot X).$$
   \item \label{actions_2} $\alpha, \beta \colon \gg_0\to \End(\Sc'(\Oc))$ are defined by  by
   $$ (\alpha(X) a) (\xi) = a(\Ad^\ast_G(-X)\xi), \quad 
   (\beta(X)a)(\Ad^\ast_G(Y)\xi_0) = a( \Ad^\ast_G(Y\cdot X)\xi_0)$$
  \end{enumerate}
  \end{definition}

With the notation above, we recall the covariance property of the Weyl-Pedersen calculus (see \cite[Thm. 3.5]{BB10c}), that is, 
\begin{equation}\label{covariance} 
\Op(\alpha(X) a) = \pi(-X) \Op(a) \pi(X), \quad \text{for all} \; X\in \gg_0.
\end{equation} 
Also, note that $\alpha$ and $\beta$ commute, that is, 
\begin{equation}\label{commutation}
\alpha(X) \beta(Y) =\beta(Y)\alpha (X), \quad X, Y\in \gg_0.
\end{equation}

\begin{definition}\label{derivatives} 
\normalfont
 For every $X\in \gg_0$, define
 
 \begin{enumerate}
  \item\label{derivatives_1}
   For every $f\in C^1(\gg_0)$  and $Y\in \gg_0$, 
   $$
   (\de \lambda (X) f)(Y) =\frac{\de}{\de t}\Big\vert_{t=0} (\lambda(tX)f)(Y) \quad  
   (\de \rho (X) f)(Y) =\frac{\de}{\de t}\Big\vert_{t=0} (\rho(tX)f)(Y). 
   $$  
   \item\label{derivatives_2} 
   For every $a\in C^1(\Oc)$ and $\xi\in \Oc$ set  
   $$
   (\de \alpha(X) a) (\xi)   =\frac{\de}{\de t}\Big\vert_{t=0} (\alpha(tX)a)(\xi), \quad
   (\de \beta(X) a) (\xi)   =\frac{\de}{\de t}\Big\vert_{t=0} (\beta(tX)a)(\xi).
    $$
  \end{enumerate}
\end{definition}

\begin{definition}
\normalfont
Recall from the Introduction that $\Diff(\Oc)$ denotes the space of 
linear differential operators on $\Oc$ that are invariant to the coadjoint action of $G$. 

We define $\Diff(\gg_0^*)$ as the space of all linear differential operators on $\gg_0^*$ 
which are pushforward through $P$ of operators in $\Diff(\Oc)$.  
Hence $D\in\Diff(\gg_0^*)$ if and only if there exists $\widetilde{D}\in\Diff(\Oc)$ 
such that $\widetilde{D}(b\circ P)=(Db)\circ P$ for every $b\in C^\infty(\gg_0^*)$. 

Similarly, $\Diff(G_0)$ is defined as the space of all linear differential operators on~$\gg_0$ 
which are pullbacks through $A_{\xi_0}$ of operators in $\Diff(\Oc)$.  
Thus $D\in\Diff(G_0)$ if and only if there exists $\widetilde{D}\in\Diff(\Oc)$ 
such that for every $a\in C^\infty(\Oc)$ we have $\widetilde{D}(a\circ A_{\xi_0})=(Da)\circ A_{\xi_0}$. 
\qed
\end{definition}

\begin{remark}\label{alg}
\normalfont
For all $X,Y\in\gg$ we have 
$$ A_{\xi_0} (X\cdot Y)=\Ad_G^*(X)(A_{\xi_0}(Y)),$$
and therefore  a linear differential operator $D$ on $\gg_0$ belongs to $\Diff(G_0)$ 
if and only if for all $b\in C^\infty(\gg_0)$ and $X\in\gg_0$ we have 
$D(L_Xb)=L_X(Db)$, where we use the left translation $L_X\colon\gg_0\to\gg_0$, $L_X(Y)=X\cdot Y$. 
Equivalently, $\Diff(G_0)$ is the space of all left-invariant linear differential operators on $\gg_0$, 
and then it follows by \cite[Ch. II, Th. 4.3]{He02} 
that there exists the isomorphism of associative algebras
$$\U((\gg_0)_{\CC})\mathop{\to}\limits^\sim\Diff(G_0), \quad u\mapsto\de\rho(u),$$
where $\U((\gg_0)_{\CC})$ stands for the universal enveloping algebra of the complexification of $\gg_0$. 
(See also \cite[Sect. 4]{CG92}.)
In particular, by using the Poincar\'e-Birkhoff-Witt theorem, we see that 
the set $\{\de\beta(X_1)^{p_1}\cdots\de\beta(X_d)^{p_d}\mid p_1,\dots,p_d\in\NN\}$ 
is a linear basis in~$\Diff(G_0)$. 
\qed
\end{remark}

Note that for every $u\in U((\gg_0)_{\CC})$ the following diagram
\begin{equation}\label{diagram}
\xymatrix{\Sc'(\Oc) 
\ar[d]_{\de\beta(u)} \ar[r]^{A^\ast_{\xi_0}} & \Sc'(\gg_0) \ar[d]^{\de\rho(u)}
\\
\Sc'(\Oc) \ar[r]^{A^\ast_{\xi_0}} & \Sc'(\gg_0) 
}
\end{equation}
is commutative, since it is clear that $\rho(X) A^*_{\xi_0} = A_{\xi_0}^\ast\beta(X)$ 
for every $X \in \gg_0$. 
It follows from Remark~\ref{alg} above that
 $a\in \Cc^\infty_b(\Oc)$ if and only if  
$$\norm{\de \beta(X_1)^{p_1} \cdots \de \beta({X_d})^{p_d} a}_{L^{\infty}(\Oc)}< \infty$$
 for all $p_1, \dots, p_d\in \NN^d$.
 Also the Fr\'echet topology on $\Cc^\infty_b(\Oc)$ can be described by the family of seminorms
$$ a\mapsto \norm{a}_{(\bar p )}= \norm{\de\beta({X}_1)^{p_1} \cdots \de\beta(X_d)^{p_d} a}_{L^{\infty}(\Oc)}.$$

By using the above remarks and explicit definitions we restate Theorem~\ref{main_introd} in the next two theorems. 
The first result is a Calder\'on-Vaillancourt type result.
\begin{theorem}\label{mainthm1}
The mapping $\Op\colon \Cc_{b}^{\infty}(\Oc) \to \BB(\Hc)$ is well defined and continuous.  
\end{theorem}

The second result provides the characterization of the space of operators given by symbols in $\Cc_{b}^{\infty}(\Oc)$. 
\begin{theorem}\label{mainthm2}
$\Cc_{b}^{\infty}(\Oc)$ is precisely the set of those symbols $a\in \Sc'(\Oc)$ such that the operator
$\Op(\de\beta(X_{1})^{p_1}\cdots \de\beta(X_d)^{p_d} a)$ is bounded for every $(p_1, \dots, p_d)\in \NN^d$.
The Fr\'echet topology of $C_b^\infty(\Oc)$ is equivalent to that defined by the family of seminorms 
$\{a\mapsto\norm{\Op(\de\beta(X_{1})^{p_1}\cdots \de\beta(X_d)^{p_d} a)}\}_{p_1, \cdots, p_d\in \NN}$.
\end{theorem} 

The proofs will be provided in Section~\ref{section-proofs}, after some preparations in the next section.

\section{Convolutions and operator calculus}\label{section-conv}

In this section we recall the definition and certain properties of a convolution between functions on $\gg_0$ and distributions on $\Oc$.  
In particular, we give a decomposition for tempered distributions on $\Oc$ in sums of suitable convolutions. 
Technically, this will replace integration by parts and  regularizations that are used in the classical case when $\gg_0$ is commutative. 
Then we use this notion of convolution to link the Weyl-Pedersen calculus and the Kato operatorial calculus (\cite{Ka76}) in this context.

 We need the involutive anti-automorphism 
 $$ U((\gg_0)_{\CC}) \to U((\gg_0)_{\CC}) , \quad u \mapsto u^{\top}$$
 uniquely determined by the condition that $X^{\top} = -X$ when $X\in \gg_0$. 
 Since 
 $\de\lambda\colon \gg_0 \to \End(C^\infty(\gg_0))$ is a Lie algebra homomorphism, it extends uniquely 
 to $\de\lambda\colon   U((\gg_0)_{\CC}) \to \End(C^\infty(\gg_0))$  which is an homomorphism of 
 associative algebras. 
 Similarly for $\rho$ and $\alpha$.  
 
 We recall that  $U((\gg_0)_{\CC})$ can be canonically identified with the space of distributions 
 on $\gg_0$ with support at $0$. 
 Thus it makes sense to consider the convolution $u\ast h$ for every $h\in \Sc'(\gg_0)$ and $u \in  U((\gg_0)_{\CC})$. 
Here convolution uses the group multiplication of $\gg_0$ and not the linear structure.

\begin{lemma}\label{conv}
 For every $u \in U((\gg_0)_{\CC})$, $\varphi\in C^\infty (\gg_0)$  and $\psi\in \Ec'(\gg_0)$ we have 
 \begin{itemize}
  \item[(a)]  $\de \lambda(u) \varphi = u\ast \varphi$, 
   \item[(b)]   $\de \rho (u) \varphi = \varphi\ast u^{\top} $, 
   \item[(c)]  $(\de \rho (u) \varphi)\ast \psi  = \varphi\ast ( \de\lambda( u^{\top})\psi) $.
  \end{itemize}
 \end{lemma}

\begin{proof}
For  (a) and (b)  see  for instance~\cite{Pe94}, while (c) follows from (a) and~(b).
\end{proof}

 \begin{lemma}\label{equivdef2}
For $\Phi \in C^\infty(\gg_0)$, the following conditions are equivalent: 
 \begin{itemize} 
 \item[(a)] 
 For every $\bar p= (p_1, \dots, p_d)\in \NN^d$, 
 $$ \de \rho(X_1)^{p_1} \cdots \de \rho(X_d)^{p_d} \Phi \in  L^\infty(\gg_0).$$
\item[(b)] For every $u \in U((\gg_0)_{\CC})$, $\de\rho(u)\Phi \in L^\infty(\gg_0)$.
\end{itemize}
\end{lemma}

\begin{proof}
 Since $X_1, \dots, X_d$ is a basis for the Lie algebra $\gg_0$, it follows by the Poincar\'e-Birkhoff-Witt theorem that 
 $\{X_1^{p_1}\cdots X_d^{p_d}\mid \bar p\in \NN^d\}$ is a basis for $U((\gg_0)_{\CC})$.
 Hence the conclusion follows at once.
 \end{proof}

\begin{remark}\label{topology_1}
\normalfont
Denote by $C_b^\infty(\gg_0)$ the space of $\Phi\in C^\infty(\gg_0)$ such that 
for every $\bar p= (p_1, \dots, p_d)\in \NN^d$, 
\begin{equation}\label{topology1_1} 
 \de \rho(X_1)^{p_1} \cdots \de \rho(X_d)^{p_d} \Phi \in  L^\infty(\gg_0).
 \end{equation}
 This is a Fr\'echet space with the topology given by the seminorms,  indexed over $\bar p \in \NN^d$, 
 $$ \Phi \mapsto \norm{ \de \rho(X_1)^{p_1} \cdots \de \rho(X_d)^{p_d} \Phi}_{L^\infty(\gg_0)}$$ 
The previous lemmas show that one can replace condition \eqref{topology1_1} by $\de\rho(u) \Phi\in L^{\infty}(\gg_0)$ for all $u\in U((\gg_0)_{\CC})$.
Moreover, the seminorms, indexed over $U((\gg_0)_{\CC})$, 
$$\Phi\mapsto \norm{\de\rho(u) \Phi}_{L^\infty(\gg_0)}$$
define an equivalent topology on $C_b^\infty(\gg_0)$.
\end{remark}

\begin{lemma}\label{decomp-algebra}
 For every $m \in \NN$, there exist finite families $\{u_j\}_{j\in J}$ in $U((\gg_0)_{\CC})$ and 
 $\{F_j\}_{j\in J}$ in $C_0^m(\gg_0)$ such that for every $h \in \Sc'(\gg_0)$, 
 $$ h = \sum\limits_{j\in J} \de \rho(u_j) h \ast F_j. $$
 \end{lemma}

\begin{proof}
Let $m\in \NN$ be fixed. 
 From Lemma~2.3 in \cite{DDJP09}, along with Lemma~\ref{conv}, we see that there exist finite families 
 $\{\tilde{u}_j\}_{j\in J}$ in $U((\gg_0)_{\CC})$ and 
 $\{F_j\}_{j\in J}$ in $C_0^m(\gg_0)$ such that 
 $ \delta = \sum_{j\in J} \tilde{u}_j \ast F_j$. 
 Hence,  for every $h  \in \Sc'(\gg_0)$, 
 $$ h = h\ast \delta =\sum\limits_{j\in J} h\ast \tilde{u}_j\ast F_j= \sum\limits_{j\in J} (\de\rho( \tilde{u}_j^\top) h)\ast F_j.$$
 Then the lemma follows with $u_j= \tilde{u}_j^\top$. 
\end{proof}

\begin{remark}\label{intertw}\normalfont
It follows by \eqref{diagram} that $a\in \Cc_b^{\infty}(\Oc)$ if and only if $A^*_{\xi_0} a\in C^\infty(\gg_0)$ and 
$$ \de \rho(X_1)^{p_1} \cdots \de \rho(X_d)^{p_d} (A_{\xi_0}^\ast a)\in L^\infty(\gg_0)$$ 
for every 
$\bar p=(p_1, \dots, p_d)\in \NN^d$.
 \end{remark}

\begin{definition}\label{convolution}
\normalfont
For $b\in \Sc'(\gg_0)$ and $a\in \Ec'(\Oc)$ we define $b\ast a\in \Sc'(\Oc)$ by 
$$ \langle b\ast a,\phi\rangle =\langle b\otimes a,\phi^{\Delta}\rangle \quad \text{for all} \; \phi\in \Sc(\Oc), $$
where $\phi^\Delta(Y, \xi) = \phi(\Ad^\ast_G (Y)\xi)$, $Y\in \gg_0$, $\xi\in \Oc$. 
\end{definition}

 \begin{remark}\label{conv-orbit}\normalfont
 Note that the convolution defined above is nothing else than the image on the orbit of the convolution on $\gg_0$, that is,  
 $$A_{\xi_0}^\ast (b\ast a) = b\ast (A_{\xi_0}^\ast a), $$
 for every $b\in \Sc'(\gg_0)$ and $a\in \Ec'(\Oc)$.
 \end{remark}

\begin{lemma}\label{decomp-orbit}
For every $m\in \NN$, there exist finite families $\{u_j\}_{j\in J} \in U((\gg_0)_{\CC})$
and $\{f_j\}_{j\in J}$ in $C_0^m(\Oc)$ such that for every 
$a\in \Sc'(\Oc)$
$$ a= \sum\limits_{j\in J} A_{\xi_0}^{\ast}(\de\beta(u_j)a) \ast f_j.$$ 
 \end{lemma}

\begin{proof}
Let $m\in \NN$ be fixed. 
By Lemma~\ref{decomp-algebra} and Remark~\ref{intertw} it follows that there are finite families 
$\{ u_j \}_{j\in J} $ in $U((\gg_0)_{\CC})$,
$\{F_j\}_{j\in J}$ in $C_0^m(\gg_0)$ such that 
$$ A_{\xi_0}^\ast a= \sum\limits_{j\in J} \de\rho(u_j) (A_{\xi_0}^\ast a) \ast F_j
= \sum\limits_{j\in J} A_{\xi_0}^\ast(\de\beta(u_j) a)  \ast F_j.$$
Thus, by Remark~\ref{conv-orbit}, the lemma follows with $f_j= (A_{\xi_0}^\ast)^{-1} F_j$. 
\end{proof}

\begin{proposition}\label{Kato}
The mapping $L^\infty(\gg_0) \times \Sg_1({\Hc}) \to \BB(\Hc)$, 
$$ (b, C) \mapsto b\{C\}= \int\limits_{\gg_0} b(X) \pi(X)^{-1} C\pi(X) \de X $$  
is well defined, bounded, and bilinear, where the integral is so-convergent.
\end{proposition}

\begin{proof}
This follows by \cite[Lemma 0.2, Ch. XIII]{Ta81} and \cite[Prop. 2.8]{BB10c}.
\end{proof}

Recall that when $a\in \Sc'(\Oc)$ and  $\phi, \psi\in \Hc_{\infty}$, we have
\begin{equation}\label{op-wig}
 \scalar{\Op(a)\phi}{\psi}= \dual{a}{\Wig(\phi, \psi)},
 \end{equation}
where the Wigner function $\Wig(\phi, \psi)\in \Sc(\Oc)$ is defined by
$$\Wig(\phi, \psi)(\xi) = \int\limits_{\gg_0} \ee^{\ie\dual{\xi}{X}} \scalar {\phi}{\pi(X)\psi}\,\de X,$$
for all $\xi\in \Oc$. 
See \cite[Cor. 2.11]{BB10c} for more details.

\begin{lemma}\label{oplemma}
Let $b\in L^\infty(\gg_0)$, and $a\in C_0(\Oc)$ be such  that $\Op(a)\in \Sg_1(\Hc)$. 
Then $\Op(b\ast a)$ extends to a bounded operator on $\Hc$, namely $\Op(b\ast a) = b\{\Op(a)\}$.
\end{lemma}

\begin{proof}
Let $\phi, \psi\in \Hc_{\infty}$. 
Then  by using \eqref{op-wig} we have
$$ 
\begin{aligned}
\scalar{\Op(b\ast a)\phi}{\psi} & = \int\limits_{\Oc} (b\ast a)(\xi) \Wig(\phi,\psi)(\xi)\, \de \xi \\
&= \int\limits_{\gg_0} \int\limits_{\Oc} b(Y) a(\Ad^\ast(-Y)\xi) \Wig(\phi, \psi) (\xi)\, \de \xi\,\de Y.
\end{aligned}
$$
Here note that 
$$ \int\limits_{\Oc} a(\Ad^\ast(-Y)\xi) \Wig(\phi, \psi) (\xi)\, \de \xi= 
\scalar{\Op(\alpha(Y)a)\phi}{\psi} = \scalar{\pi(-Y) \Op(a) \pi(Y)\phi}{\psi},$$
by the covariance equality \eqref{covariance}. 
The lemma follows now from Proposition~\ref{Kato}. 
\end{proof}

\section{Proofs of the main results}\label{section-proofs}

We recall that for proving Theorem~\ref{main_introd} it is enough to prove Theorem~\ref{mainthm1} and Theorem~\ref{mainthm2}.

\begin{proposition}\label{traceclass}
There exists $m\in \NN$ such that when $f\in \Cc_0^m(\Oc)$ the operator $\Op(f)$ is trace class. 
\end{proposition}

\begin{proof}
Recall (see \cite{Pe94}) that the mapping $\Op\colon \Sc(\Oc) \to \BB(\Hc)_{\infty}$ is an isomorphism. 
On the other hand, $\BB(\Hc)_\infty$ is continuously embedded in $\Sg_1(\Hc)$ 
(see \cite{BB10b}). 
Thus there exist $m, l\in \NN$ and $C>0$ such that 
$$\norm{\Op(\phi)}_{\Sg_1(\Hc)}\le C p_{m, l} (\phi),$$
for every $\phi\in \Sc(\Oc)$, 
where $p_{m, l}$ is the seminorm on $\Sc(\Oc)$ that corresponds, via the global chart $P$, to the seminorm
$\phi\mapsto \sup|(1+|\xi|^2)^{l/2}(1-\Delta)^{m/2} \phi(\xi)| $. 
Since $\Cc^m_0(\Oc)$ is contained in the closure of $\Sc(\Oc)$ in this seminorm, we get that  
$\Op(f)$ can be approximated in the $\Sg_1(\Hc)$ norm  by a sequence of trace class operators, whenever 
$f\in \Cc^m_0(\Oc)$.
This proves that $\Op(f)\in \Sg_1(\Hc)$.
\end{proof}

\begin{proof}[Proof of Theorem~\ref{mainthm1}]
First use Proposition~\ref{traceclass} to find $m\in \NN$ 
large enough such that $\Op(f)$ is trace class whenever $f\in \Cc_0^{m}(\Oc)$. 
Then by Lemma~\ref{decomp-orbit} there exist finite families $\{u_j\}_{j\in J}$ and $\{f_j\}_{j\in J}$
such that for every $a\in \Sc'(\Oc)$ one has 
$$ a  = \sum\limits_{j\in J} A_{\xi_0}^\ast (\de \beta(u_j) a) \ast f_j.$$ 
Assume now that $a\in \Cc_b^\infty(\Oc)$. 
Then $b_j= A_{\xi_0}^{\ast} (\de \beta(u_j) a) \in L^\infty(\gg_0)$ (see Remarks~\ref{intertw} and \ref{topology_1}).
It follows by Lemma~\ref{oplemma} and Proposition~\ref{Kato} that 
$$ \Op(a) =\sum\limits_{j\in J} b_j\{\Op(f_j)\}, $$
$\Op(a)$ is bounded, and there are constants $C, C'>0$ such that 
$$ \norm{\Op(a)}_{\BB(\Hc)} \le C \sum\limits_{j\in J} \norm{b_j}_{L^\infty(\gg_0)} \norm{\Op(f_j)}_{\Sg_1(\Hc)}
\le C' \sum\limits_{j\in J} \norm{\de \beta_j(u_j) a}_{L^\infty(\Oc)}.$$
This, together with Remark~\ref{intertw}, concludes the proof.
\end{proof}

In order to prove Theorem~\ref{mainthm2} we need a lemma. 

\begin{lemma}\label{Bony}
Let $B$ be a Banach space of distributions continuously embedded into $\Sc'(\Oc)$, and such that
$\alpha(X) a\in B$ and $\norm{\alpha(X)a}_{B}=\norm{a}_{B}$  whenever $a\in B$  and $X\in \gg_0$.
Then for every $u\in U(( \gg_0)_{\CC})$ there is a constant $C$ and a finite family $\{u_j\}_{j\in J}$ such that 
$$ \norm{\de\beta(u) a}_{L^\infty} \le C\sum\limits_{j\in J} \norm{\de\beta(u_j) a}_{B}$$
for every $a\in B$ with  $\de\beta(u') a\in B$, for all $ u'\in U(( \gg_0)_{\CC})$. 
\end{lemma}

\begin{proof}
The unit ball in $B$ is bounded in $\Sc'(\Oc)$,
therefore  there exist $m, k\in \NN$  and   $C>0$ such that for every 
$a\in B$ and $\varphi\in \Sc(\Oc)$  
\begin{equation}\label{Bony_1}
 \vert \dual{a}{\varphi}\vert \le C \norm{a}_{B} p_{m, k} (\varphi), 
\end{equation}
where $p_{m, k}(\varphi)$ is a seminorm on $\Sc(\Oc)$ as in the proof of Proposition~\ref{traceclass}. 
It follows that $a$ extends to functions in $\Cc_0^m(\Oc)$ and,  when $K$ is a compact subset of $\Oc$, there is $C=C(K)>0$ such that
\begin{equation}\label{Bony_2}
 \vert \dual{a}{\varphi}\vert \le C \norm{a}_{B} \norm{\varphi}_{\Cc^m(K)}, 
\end{equation}
whenever $\varphi\in \Cc_0^\infty(\Oc)$ has support in $K$.

For $f\in \Cc_0^m(\Oc)$,
let  $\check{f}\in \Cc_0^m(\Oc)$  be the  function defined by 
$$\check{f}(\Ad^\ast(X)\xi_0)= f(\Ad^\ast(-X)\xi_0).$$
Then note that when $a\in \Sc'(\Oc)$ and $f\in \Cc_0^\infty(\Oc)$  one can write
$$
((A_{\xi_0}^\ast a)\ast f)(\Ad^\ast(X)\xi_0)= \dual{A_{\xi_0}^\ast a(X\cdot (\cdot))}{\check{f}(\Ad^\ast(\cdot)\xi_0)}
= \dual{\alpha(-X)a}{\check{f}}
$$
for every $X\in \gg_0$. 
When $a\in B$, since $\alpha(-X)a\in B$, the above equality extends to $f\in\Cc_0^m(\Oc)$, by \eqref{Bony_2}.
It follows that if  $K$ is a compact in $\Oc$, there is $C=C(K)>0$ such that  for every $a\in B$ and $X\in \gg_0$, 
\begin{equation}\label{Bony_3} 
\vert ((A_{\xi_0}^\ast a)\ast f)(\Ad^\ast(X)\xi_0)\vert= 
\vert \dual{\alpha(-X)a}{\check{f}}\vert\le C \norm{a}_{B} \norm{\check{f}}_{C^m(K)},
\end{equation}
whenever $f\in \Cc_0^\infty(\Oc)$ has support in $K$.

We now use Lemma~\ref{decomp-orbit} for $m$ above. 
Thus, for some finite families $\{u_j\}_{j\in J}$ in 
$U((\gg_0)_{\CC})$ and $\{f_j\}_{j\in J}$ in $C_0^m(\Oc)$, independent on $a\in B$, we may write
$$a(\Ad^\ast(X)\xi_0)= \sum\limits_{j\in J}( (A_{\xi_0}^\ast(\de \beta(u_j) a))\ast f_j)(\Ad^\ast(X)\xi_0).
$$
From \eqref{Bony_3} we get 
$$\norm{a}_{L^\infty} \le C\sum\limits_{j\in J} \norm{\de\beta(u_j)a}_B,  
$$
where $C$ is independent on $a$. 

To get the estimates for $\de\beta(u) a$ in the statement, we use the above inequality for $a$ replaced by $\de\beta(u)a$. 
This concludes the proof of the lemma.
\end{proof}

\begin{proof}[Proof of Theorem~\ref{mainthm2}]
We use the above Lemma~\ref{Bony} for the  Banach space 
$$ B=\{a\in \Sc'(\Oc) \mid \Op(a)\in \BB(\Hc)\}$$
with the norm $\norm{a}_{B}=\norm{\Op(a)}_{\BB(\Hc)}$. 
 
It remains to notice that when $a\in B$, and $X\in \gg_0$, $\alpha(X)a\in B$ and has the same norm as $a$,
since  $\Op(\alpha(X)a) =\pi(X) \Op(a)\pi(-X)$ by~\eqref{covariance}.
\end{proof}

\begin{lemma}\label{Bony_Lp}
If $p\in[1,\infty)$ and $f\in\Cc_0^m(\Oc)$, then there exists a constant $M>0$ 
such that for every $c\in\Sc'(\Oc)$  
we have  $\Vert A_{\xi_0}^*(c)\ast f\Vert_{L^p(\Oc)}\le M\Vert \Op(c)\Vert_{\Sg_p(\Hc)}$. 
\end{lemma}

\begin{proof}
First note that $\Cc_0^\infty(\Oc)$ is dense $\Sc(\Oc)$, 
which in turn is dense in the Banach space 
$B:=\{c\in\Sc'(\Oc)\mid\Vert c\Vert_B:=\Vert \Op(c)\Vert_{\Sg_p(\Hc)}<\infty\}$. 
It then follows that it suffices to obtain an estimate of the form 
\begin{equation}\label{Bony_Lp_eq1}
(\forall c,\varphi\in\Cc_0^\infty(\Oc))\quad \Bigl\vert\int\limits_{\Oc}(A_{\xi_0}^*(c)\ast f)\varphi\Bigr\vert\le M \Vert \Op(c)\Vert_{\Sg_p(\Hc)}
\Vert\varphi\Vert_{L^q(\Oc)}
\end{equation}
where the constant $M>0$ depends only on $p$ and $f$, 
where $\frac{1}{p}+\frac{1}{q}=1$. 
By using Remark~\ref{conv-orbit} we get  
$$\begin{aligned}
\int\limits_{\Oc}(A_{\xi_0}^*(c)\ast f)\varphi
&=\int\limits_{\gg_e}A_{\xi_0}^*((A_{\xi_0}^*(c)\ast f)\varphi) \\
&=\int\limits_{\gg_e}((A_{\xi_0}^*(c)\ast A_{\xi_0}^*(f))A_{\xi_0}^*(\varphi) \\
&=\iint\limits_{\gg_e\times\gg_e}(A_{\xi_0}^*(c))(Y)(A_{\xi_0}^*(f))((-Y)\ast X)(A_{\xi_0}^*(\varphi))(X)\de X\de Y \\
&=\int\limits_{\gg_e}A_{\xi_0}^*(c) (A_{\xi_0}^*(\varphi)\ast A_{\xi_0}^*(\check{f})) \\
&=\int\limits_{\Oc}c (A_{\xi_0}^*(\varphi)\ast \check{f}).
\end{aligned} $$
Furthermore, by using \cite[Th. 4.1.4]{Pe94} we get 
$$\begin{aligned}
\Bigl\vert\int\limits_{\Oc}(A_{\xi_0}^*(c)\ast f)\varphi\Bigr\vert
&=\vert\Tr(\Op (c) \Op(A_{\xi_0}^*(\varphi)\ast \check{f}))\vert \\
&\le\Vert \Op(c)\Vert_{\Sg_p(\Hc)} \Vert \Op(A_{\xi_0}^*(\varphi)\ast \check{f}))\Vert_{\Sg_q(\Hc)} \\
&=\Vert \Op(c)\Vert_{\Sg_p(\Hc)} \Vert A_{\xi_0}^*(\varphi)\Op\{\check{f})\}\Vert_{\Sg_q(\Hc)}, 
\end{aligned} $$
where the latter equality follows by Lemma~\ref{oplemma}. 
Now, by using \cite[Lemma 3.3]{Ar08} we get \eqref{Bony_Lp_eq1}, 
and this completes the proof. 
\end{proof}

By using the above lemma and the $L^p$ version of Proposition~\ref{Kato} obtained by the similar reasoning and 
\cite[Lemma 3.3]{Ar08}, we get the following theorem.

\begin{theorem}\label{main_Schatten}
Let $G$ be a connected, simply connected, nilpotent Lie group whose generic coadjoint orbits are flat. 
Let $\Oc$ be such an orbit with a corresponding unitary irreducible representation $\pi\colon G\to\BB(\Hc)$, 
and $p\in[1,\infty)$.  
Then the following assertions are equivalent for $a\in C^\infty(\Oc)$: 
\begin{enumerate}
\item For every $D\in\Diff(\Oc)$ we have $Da\in L^p(\Oc)$. 
\item For every $D\in\Diff(\Oc)$ we have $\Op(Da)\in\Sg_p(\Hc)$. 
\end{enumerate}
Moreover, if we denote by $\Cc^{\infty, p}(\Oc)$ the space of all symbols satisfying the above conditions, 
then the Weyl-Pedersen calculus defines a continuous linear map 
$$\Op\colon \Cc^{\infty,p}(\Oc)\to\Sg_p(\Hc)$$
if $\Cc^{\infty, p}(\Oc)$ carries the Fr\'echet topology that can be defined by any of 
the families of seminorms $\{a\mapsto\Vert Da\Vert_{L^\infty(\Oc)}\}_{D\in\Diff(\Oc)}$ and $\{a\mapsto\norm{\Op(Da)}_{\Sg_p(\Hc)}\}_{D\in\Diff(\Oc)}$. 
\end{theorem}

\section{Invariant differential operators on coadjoint orbits}\label{computations}

In this section we compute invariant differential operators on a flat coadjoint orbit. 
We first make some general considerations centred on the idea of affine coadjoint action, and
then we illustrate the method by two examples of 5-dimensional Lie algebras. 
 
\begin{remark}
\normalfont
Define 
$$\omega\colon\gg_0\times\gg_0\to\RR,\quad \omega(X,Y)=\langle \xi_0,[X,Y]_{\gg}\rangle$$
and 
$$\chi:=P\circ A_{\xi_0}\colon \gg_0\to\gg_0^*,\quad \chi(X)=(\Ad_G^*(X)\xi_0)\vert_{\gg_0}.$$
Then it is straightforward to check the following properties: 
\begin{enumerate}
\item The functional $\omega$ gives a symplectic structure on the Lie algebra $\gg_0$.  
That is, it is a nondegenerate scalar 2-cocycle on $\gg_0$, 
in the sense that it is skew-symmetric, the linear mapping 
$$\omega^\#\colon\gg_0\to\gg_0^*, \quad \omega^\#(X)=\omega(\cdot,X)$$
is invertible, and 
$$(\forall X,Y,Z\in\gg_0)\quad \omega(X,[Y,Z])+\omega(Y,[Z,X])+\omega(Z,[X,Y])=0. $$
Since $\omega$ is skew-symmetric, the above property is equivalent to the fact that 
$$(\forall Y,Z\in\gg_0)\quad  \omega^\#([Y,Z])=(\ad_{\gg_0}^*Y)\omega^\#(Z)-(\ad_{\gg_0}^*Z)\omega^\#(Y),$$
hence the function $\omega^\#$ 
is a 1-cocycle of the Lie algebra $\gg_0$ 
with values in its coadjoint representation. 
\item The function $\chi$ is a smooth 1-cocycle of the Lie group $G_0$ 
with values in its coadjoint representation, 
which means that 
$$(\forall X,Y\in\gg_0)\quad \chi(X\cdot Y)=\chi(X)+\Ad_{G_0}^*(X)\chi(Y). $$
\end{enumerate}
To see the relation between the properties of $\chi$ and $\omega$,  
just note that $\chi'_0=\omega^\#$ and use the connection between 
the Lie group cohomology and the Lie algebra cohomology; 
see for instance \cite[Th. 10.1]{CE48} or~\cite[App. B]{Ne04}. 
\qed
\end{remark}

Note that all of the associative algebras of linear differential operators 
$$\Diff(\Oc),\  \Diff(\gg_0^*),\ \Diff(G_0) $$
are mutually isomorphic and are also isomorphic to the universal enveloping algebra $\U((\gg_0)_{\CC})$, 
by Remark~\ref{alg}. 
We now proceed to developing the ingredients for  
a description of the differential operators in~$\Diff(\gg_0^*)$. 
Recall that these are nothing than the operators in $\Diff(\Oc)$ viewed in 
the global chart $P\colon\Oc\to\gg_0^*$, $P(\xi)=\xi\vert_{\gg_0}$. 

\begin{definition}\label{aff}
\normalfont
The \emph{affine coadjoint action of $G_0$} is the mapping
$$\gamma\colon G_0\times\gg_0^*\to\gg_0^*,\quad 
(X,\eta)\mapsto\gamma(X)\eta:=\Ad_{G_0}^*(X)\eta+\chi(X), $$
while the mapping 
$$\dot\gamma_\omega\colon \gg_0\times\gg_0^*\to\gg_0^*,\quad 
(X,\eta)\mapsto\dot\gamma_\omega(X)(\eta):=\ad_{G_0}^*(X)\eta+\omega^\#(X) $$
is called the \emph{affine coadjoint action of $\gg_0$}.

Moreover for every $f\in C^\infty(\gg_0^*)$ and $\eta\in\gg_0^*$ 
we define $\nabla_\eta f\in\gg_0$ to be such that $f'_\eta=\langle\cdot,\nabla_\eta f\rangle_{\gg_0^*,\gg_0}\colon\gg_0^*\to\RR$, 
where $\langle \cdot,\cdot\rangle_{\gg_0^*,\gg_0}\colon\gg_0^*\times\gg_0\to\RR$ 
denotes the duality pairing.
Then for every $X\in\gg_0$ we define the linear differential operator 
$\gamma_{\omega}(X)\colon C^\infty(\gg_0^*)\to C^\infty(\gg_0^*)$ 
that acts on arbitrary $f\in C^\infty(\gg_0^*)$ by 
$$(\forall\eta\in\gg_0^*)\quad 
(\gamma_{\omega}(X)f)(\eta)
=\langle f'_\eta,\dot\gamma_\omega(X)(\eta)\rangle_{\gg_0^*,\gg_0}.$$
Hence $\gamma_{\omega}(X)$ is just the first-order differential operator 
corresponding to the fundamental vector field generated by $X$ in the infinitesimal affine coadjoint action of~$\gg_0$. 
\qed
\end{definition}

\begin{remark}\label{expl}
\normalfont
Note that for all $X\in\gg_0$ and  $f\in C^\infty(\gg_0^*)$ we have 
$$(\forall\eta\in\gg_0^*)\quad 
(\gamma_{\omega}(X)f)(\eta)
=\omega(X,\nabla_\eta f)-\langle \eta,[X,\nabla_\eta f]\rangle_{\gg_0^*,\gg_0}$$
and we thus see that the coefficients of $\gamma_{\omega}(X)$ are polynomials of degree at most~1. 
More precisely, if we denote by $\partial_1,\dots,\partial_d$ the partial derivatives in $ C^\infty(\gg_0^*)$ 
with respect to the coordinates 
$$(\eta_1,\dots,\eta_d)=(\langle\eta,X_1\rangle_{\gg_0^*,\gg_0},\dots,\langle\eta,X_d\rangle_{\gg_0^*,\gg_0})$$
defined by the basis of $\gg_0^*$ which is 
dual to the basis $X_1,\dots,X_d$ of~$\gg_0$, 
then we have for arbitrary $f\in C^\infty(\gg_0^*)$,
$$(\forall\eta\in\gg_0^*)\quad \nabla_\eta f=(\partial_1 f)(\eta)X_1+\cdots+(\partial_d f)(\eta)X_d $$
hence 
$$(\forall\eta\in\gg_0^*)\quad 
(\gamma_{\omega}(X)f)(\eta)
=\sum_{j=1}^d\omega(X,X_j)(\partial_j f)(\eta) - 
\sum_{j=1}^d\langle \eta,[X,X_j]\rangle_{\gg_0^*,\gg_0}(\partial_j f)(\eta).$$ 
Now let us define the linear functionals $c_{jk}\colon\gg_0\to\RR$ for $1\le k<j\le d$ such that 
$$(\forall j\in\{1,\dots,d\})(\forall X\in\gg_0)\quad 
[X,X_j]=\sum_{1\le k<j}c_{jk}(X)X_k.$$
Then 
$$(\forall X\in\gg_0)\quad \gamma_{\omega}(X)=\sum_{j=1}^d\omega(X,X_j)\partial_j -
\sum_{1\le k<j}c_{jk}(X)\eta_k\partial_j.$$
\qed
\end{remark}

We now motivate the terminology used in Definition~\ref{aff}.

\begin{proposition}\label{aff_prop}
For every $X\in\gg_0$ the diagram 
$$ \xymatrix{\Oc \ar[r]^{\Ad_G^*(X)} \ar[d]_{P}& \Oc \ar[d]^{P}\\
\gg_0^* \ar[r]^{\gamma(X)}& \gg_0^*}$$
is commutative. 
\end{proposition}

\begin{proof}
For arbitrary $X,Y\in\gg_0$ and $\xi\in\Oc$ we have 
$$\langle P(\Ad_G^*(X)\xi),Y\rangle_{\gg_0^*,\gg_0} 
=\langle \Ad_G^*(X)\xi,Y\rangle
=\langle\xi,\Ad_G(-X)Y\rangle. $$
Moreover, if we denote by $\pr_{\zg}$ and $\pr_{\gg_0}$  the Cartesian projections 
corresponding to the decomposition $\gg=\zg\dotplus\gg_0$, then we have 
$$\begin{aligned}
\Ad_G(-X)Y
&=\pr_{\zg}(\Ad_G(-X)Y)+\pr_{\gg_0}(\Ad_G(-X)Y) \\
&=\langle\xi_0,\Ad_G(-X)Y\rangle X_0+\Ad_{G_0}(-X)Y. 
\end{aligned}$$
By using the preceding equality we then get 
$$\begin{aligned}
\langle P(\Ad_G^*(X)\xi),Y\rangle_{\gg_0^*,\gg_0} 
&=\langle\xi,\Ad_G(-X)Y\rangle \\
&=\langle\xi,\langle\xi_0,\Ad_G(-X)Y\rangle X_0+\Ad_{G_0}(-X)Y\rangle \\
&=\langle\xi_0,\Ad_G(-X)Y\rangle \langle\xi, X_0\rangle+\langle\xi,\Ad_{G_0}(-X)Y\rangle \\
&=\langle\xi_0,\Ad_G(-X)Y\rangle +\langle\xi,\Ad_{G_0}(-X)Y\rangle\\
&=\langle\Ad_G^*(X)\xi_0,Y\rangle +\langle P(\xi),\Ad_{G_0}(-X)Y\rangle_{\gg_0^*,\gg_0}\\
&=\langle\chi(X),Y\rangle_{\gg_0^*,\gg_0} +\langle \Ad_{G_0}^*X(P(\xi)),Y\rangle_{\gg_0^*,\gg_0}
\end{aligned}$$
and this concludes the proof since $Y\in\gg_0$ is arbitrary. 
\end{proof}

\begin{corollary}\label{aff_cor1}
If $D\colon C^\infty(\gg_0^*)\to C^\infty(\gg_0^*)$ is a linear differential operator, 
then we have $D\in\Diff(\gg_0^*)$ 
if and only if for every $X\in\gg_0$ we have $[D,\gamma_{\omega}(X)]=0$. 
\end{corollary}

\begin{proof}
Use Proposition~\ref{aff_prop} and Remark~\ref{alg}. 
\end{proof}

\begin{corollary}\label{aff_cor2}
The set $\Diff(\gg_0^*)$ is equal to the unital associative algebra 
generated by the first-order linear differential operators  $D\colon C^\infty(\gg_0^*)\to C^\infty(\gg_0^*)$ 
satisfying the condition $[D,\gamma_{\omega}(X)]=0$ for arbitrary $X\in\gg_0$. 
\end{corollary}

\begin{proof}
Use Corollary~\ref{aff_cor1} and Remark~\ref{alg}. 
\end{proof}

\begin{example}\label{N5N3}
\normalfont
Let us assume that the Lie algebra $\gg$ is 5-dimensional and is defined by the commutation relations 
$$[X_4,X_3]_{\gg}=X_1,\  [X_4,X_1]_{\gg}=[X_3,X_2]_{\gg}=X_0.$$
Then the Lie algebra $\gg_0=\spa\{X_1,X_2,X_3,X_4\}$ is defined by the commutation relation 
$$[X_4,X_3]=X_1.$$ 
The skew-symmetric bilinear functional $\omega\colon\gg_0\times\gg_0\to\RR$ is 
defined by  
$\omega(X_1,X_4)=\omega(X_2,X_3)=1$ and $\omega(X_i,X_j)=0$ if $1\le i<j\le 4$ and 
$(i,j)\not\in\{(1,4),(2,3)\}$. 
Therefore, by using Remark~\ref{expl}, 
 we get 
$$\begin{aligned}
\gamma_\omega(X_1)&=\partial_4, \\
\gamma_\omega(X_2)&=\partial_3, \\
\gamma_\omega(X_3)&=-\partial_2+\eta_1\partial_4, \\
\gamma_\omega(X_4)&=-\partial_1-\eta_1\partial_3.
\end{aligned} $$
To compute $\Diff(\gg_0^*)$
let $D=\sum\limits_{j=1}^4a_j\partial_j$ be a first-order differential operator with polynomial coefficients 
on $\gg_0^*$ with $[\gamma_\omega(X_k),D]=0$ for $k=1,2,3,4$. 
In order to determine the conditions to be satisfied by the coefficients of $D$, 
we use the following commutation formula for first-order partial differential operators, 
which can be verified directly: 
\begin{equation}\label{N5N3_eq1}
[b_i\partial_i,a_j\partial_j]=b_i(\partial_i a_j)\partial_j-a_j(\partial_j b_i)\partial_i.
\end{equation}
We then  see, by direct computation and determining the coefficients $a_j$,  that the space of first-order linear differential operators 
in $\Diff(\gg_0^*)$ is the 4-dimensional vector space 
$$\{c_1\partial_1+c_2\partial_2+(c_1\eta_1+c_3)\partial_3+(-c_1\eta_2+c_4)\partial_4\mid c_1,c_2,c_3,c_4\in\RR\}$$
that is, 
$\spa\{\partial_2,\partial_3,\partial_4,\partial_1+\eta_1\partial_3-\eta_2\partial_4\}$. 
Then Corollary~\ref{aff_cor2} 
 shows that $\Diff(\gg_0^*)$ is the unital associative algebra of differential operators 
generated by this set. 

We recall from \cite[page 18]{Pe88} that 
the unitary irreducible representation of $G=(\gg,\cdot)$ associated with the flat coadjoint orbit discussed here is 
$\pi\colon G\to\Bc(L^2(\RR^2))$, with the derivative 
$\de\pi\colon\gg\to\Lc(\Sc(\RR^2))$ given by 
\begin{align}
\de\pi(X_0)&= \ie\1\nonumber \\
\de\pi(X_1)&= \ie t_1 \nonumber\\
\de\pi(X_2)&= \ie t_2 \nonumber\\
\de\pi(X_3)&= \partial_{t_2} \nonumber\\
\de\pi(X_4)&= \partial_{t_1}-\ie t_1t_2\nonumber
\end{align}
where $(t_1,t_2)$ is the generic point in $\RR^2$ and the operators $\de\pi(X_1)$ and $\de\pi(X_2)$ 
are the operators of multiplication by the coordinate functions as indicated above. 
\qed
\end{example}

\begin{example}\label{N5N6}
\normalfont
Let us assume that the Lie algebra $\gg$ is again 5-dimensional and is defined this time by the commutation relations 
$$[X_4,X_3]_{\gg}=X_2,\  [X_4,X_2]_{\gg}=X_1,\ [X_4,X_1]_{\gg}=[X_3,X_2]_{\gg}=X_0.$$
Then the Lie algebra $\gg_0=\spa\{X_1,X_2,X_3,X_4\}$ is defined by the commutation relations 
$$[X_4,X_3]=X_2,\ [X_4,X_2]=X_1.$$ 
The skew-symmetric bilinear functional $\omega\colon\gg_0\times\gg_0\to\RR$ is 
defined by  
$\omega(X_1,X_4)=\omega(X_2,X_3)=1$ and $\omega(X_i,X_j)=0$ if $1\le i<j\le 4$ and 
$(i,j)\not\in\{(1,4),(2,3)\}$. 
By using Remark~\ref{expl}, 
 we get 
$$\begin{aligned}
\gamma_\omega(X_1)&=\partial_4, \\
\gamma_\omega(X_2)&=\partial_3+\eta_1\partial_4, \\
\gamma_\omega(X_3)&=-\partial_2+\eta_2\partial_4, \\
\gamma_\omega(X_4)&=-\partial_1-\eta_1\partial_2-\eta_2\partial_3.
\end{aligned} $$
We again wish to compute $\Diff(\gg_0^*)$ by using Corollary~\ref{aff_cor2}. 
So let $D=\sum\limits_{j=1}^4a_j\partial_j$ be a first-order differential operator with polynomial coefficients 
on $\gg_0^*$ such that $[\gamma_\omega(X_k),D]=0$ for $k=1,2,3,4$. 
It  eventually follows, again by a direct computation and solving some simple first order 
differential equations satified by $a_j$,  
that the space of first-order linear differential operators 
in $\Diff(\gg_0^*)$ is the 4-dimensional vector space 
consisting of the operators 
$$\begin{aligned} 
&
c_1\partial_1+(c_1\eta_1+c_2)\partial_2+(\frac{c_1}{2}\eta_1^2+c_2\eta_1+c_3)\partial_3 \\
&
+(c_1\eta_3-c_1\eta_1\eta_2-c_2\eta_2+\frac{c_1}{3}\eta_1^3+\frac{c_2}{2}\eta_1^2+c_4)\partial_4
\end{aligned}$$  
with $c_1,c_2,c_3,c_4\in\RR$,
that is, 
$$\spa\{\partial_3,\partial_4,
\partial_2+\eta_1\partial_3+(-\eta_2+\frac{\eta_1^2}{2})\partial_4, 
\partial_1+\eta_1\partial_2+\frac{\eta_1^2}{2}\partial_3 
+(\eta_3-\eta_1\eta_2+\frac{\eta_1^3}{3})\partial_4\}.$$
Then Corollary~\ref{aff_cor2} 
 shows that $\Diff(\gg_0^*)$ is the unital associative algebra of differential operators 
generated by this set. 

We recall from \cite[page 29]{Pe88} that 
the unitary irreducible representation of $G=(\gg,\cdot)$ associated with the flat coadjoint orbit discussed here is 
$\pi\colon G\to\Bc(L^2(\RR^2))$, with the derivative 
$\de\pi\colon\gg\to\Lc(\Sc(\RR^2))$ given by 
\begin{align}
\de\pi(X_0)&= \ie\1\nonumber \\
\de\pi(X_1)&= \ie t_1 \nonumber\\
\de\pi(X_2)&= \ie t_2 \nonumber\\
\de\pi(X_3)&= \partial_{t_2} \nonumber\\
\de\pi(X_4)&= \partial_{t_1}+t_1\partial_{t_2}-\frac{\ie}{2} t_2^2\nonumber
\end{align}
where, just as in Example~\ref{N5N3}, 
$(t_1,t_2)$ is the generic point in $\RR^2$ and the operators $\de\pi(X_1)$ and $\de\pi(X_2)$ 
are the operators of multiplication by the coordinate functions as indicated above. 
\qed
\end{example}

\section{Application to convolution operators on the Heisenberg group}\label{final}

\subsection{Notation for the Heisenberg group}\hfill


$\bullet$ 
The Heisenberg algebra 
$\hg:=\hg_{2m+1}=\spa\{Z,Y_1,\dots,Y_m,X_1,\dots,X_m\}$, 
with 
\begin{equation}\label{can}
[Y_j,X_j]=Z
\end{equation} 
for $j=1,\dots,m$, and the group $H=(\hg, \cdot) $ given by
$$(\forall X,Y\in\hg)\quad X\cdot Y:=X+Y+\frac{1}{2}[X,Y]. $$

$\bullet$
The dual to the Heisenberg algebra: 

$\hg^*:=\hg_{2m+1}^*=\spa\{Z^*,Y_1^*,\dots,Y_m^*,X_1^*,\dots,X_m^*\}$

$\bullet$
The semidirect product: 

\begin{itemize}

\item[--] Consider $\Fc:=\hg^*+\RR\1\subseteq\Ci(\hg)$, where $\1$ denotes the constant function equal to 1 on $\hg$


\item[--] Consider the action of $H $ in $\Fc$ given by  $\lambda\colon\hg\to\End(\Fc)$, $(\lambda(X)\phi)(Y):=\phi((-X)\cdot Y)$.
Note that for all $\xi\in\hg^*$ and $X\in \hg$ we have 
\begin{equation}\label{deriv}
        \de\lambda(X)\xi=-\langle\xi,X\rangle-\frac{1}{2}\xi\circ(\ad_{\hg}X), 
\end{equation}         
since
$$(\de\lambda(X)\xi)(Y)=\frac{\de}{\de t}\Big\vert_{t=0}\langle\xi,(-tX)\cdot Y\rangle
=\frac{\de}{\de t}\Big\vert_{t=0}\langle\xi,-tX+ Y-\frac{t}{2}[X,Y]\rangle. $$
\item[--]
For the Heisenberg group $H=(\hg,\cdot)$ we can consider the semidirect product 
$G:=\Fc\rtimes_{\lambda}H$, whose Lie algebra is $\gg:=\Fc\rtimes_{\de\lambda}\hg$,  
with the bracket given by 
$$[(\phi_1,X_1),(\phi_2,X_2)]=(\de\lambda(X_1)\phi_2-\de\lambda(X_2)\phi_1,[X_1,X_2]) $$
for $\phi_1,\phi_2\in\Fc=\hg^*+\RR\1$ and $X_1,X_2\in\hg$. 
This implies in particular that $\gg$ is a 3-step nilpotent Lie algebra with the 1-dimensional center  
$\zg:=\RR(\1,0)$, hence this center consists of the constant functions on $\hg$. 
\item[--]
 It follows by the above remarks that if we define $\gg_0:=\gg/\zg$, then we obtain 
the symplectic 2-step nilpotent Lie algebra 
$$\begin{aligned}
\gg_0
:=&\hg^*\dotplus\hg \\
=& \spa\{Z^*,Y_1^*,\dots,Y_m^*,X_1^*,\dots,X_m^*,Z,Y_1,\dots,Y_m,X_1,\dots,X_m\}
\end{aligned}$$
(predual to a flat coadjoint orbit in the 3-step nilpotent Lie group $G$ with 1-dimensional center)
 with the bracket given by 
$$[\xi_1+X_1,\xi_2+X_2]=\frac{1}{2}\xi_1\circ(\ad_{\hg}X_2)-\frac{1}{2}\xi_2\circ(\ad_{\hg}X_1)+[X_1,X_2] $$
for $\xi_1,\xi_2\in\gg^*$ and $X_1,X_2\in\hg$.
\end{itemize}

We note the following commutation relations in the Lie algebra $\gg_0$: 
\begin{align}
\label{commutZX}
 [Z^*,X_j]&=\frac{1}{2}Z^*\circ(\ad_{\hg}X_j)=-\frac{1}{2}Y_j^* \\
 \label{commutZY}
 [Z^*,Y_j]&=\frac{1}{2}Z^*\circ(\ad_{\hg}Y_j)=\frac{1}{2}X_j^*
\end{align}
which follow from \eqref{can} and \eqref{deriv}.


$\bullet$
The generic point in $\gg_0^*=\hg_{2n+1}\dotplus\hg_{2n+1}^*$: 
$$(\zeta,\eta_1,\dots,\eta_m,\xi_1,\dots,\xi_m,\zeta^*,\eta_1^*,\dots,\eta_m^*,\xi_1^*,\dots,\xi_m^*)$$

$\bullet$
The symplectic form on $\gg_0$: 

$\omega\colon\gg_0\times\gg_0\to\RR$, \quad 
$(\forall V_0\in\{Z,Y_1,\dots,Y_m,X_1,\dots,X_m\})\quad \omega(V_0^*,V_0)=1$

\begin{lemma}
A basis of
fundamental vector fields generated by the affine coadjoint action of $G_0$ are given by
\begin{align}
\label{Xstar}
\gamma_\omega(X_k^*) & =\partial_{\xi_k} \\
\label{Ystar}
\gamma_\omega(Y_k^*) & =\partial_{\eta_k} \\
\label{Zstar}
\gamma_\omega(Z^*) & =\partial_{\zeta}
-\frac{1}{2}(\sum_{k=1}^m\xi_k^*\partial_{\eta_k}-\sum_{k=1}^m\eta_k^*\partial_{\xi_k}) \\
\label{X}
\gamma_\omega(X_k)& 
=-\partial_{\xi_k^*}-(\frac{\eta_k^*}{2}\partial_{\zeta^*}-\zeta\partial_{\eta_k})\\
\label{Y}
\gamma_\omega(Y_k)&=-\partial_{\eta_k^*}-(-\frac{\xi_k^*}{2}\partial_{\zeta^*}+\zeta\partial_{\xi_k}) \\
\label{Z}
\gamma_\omega(Z) &=-\partial_{\zeta^*} 
\end{align}
for $j=1,\dots,m$.
\end{lemma} 
\begin{proof}
Recall from Remark~\ref{expl} that 
if $V_1,\dots,V_d$ is a Jordan-H\"older basis in a symplectic nilpotent Lie algebra $(\gg,\omega)$ 
and $X\in\gg$ with
$[X,V_j]-\sum\limits_{k=1}^{j-1}c_{jk}(X)V_k=0$ for $j=1,\dots,m$, 
then 
the corresponding infinitesimal generator of the affine coadjoint action is 
$$\gamma_\omega(X)=\sum_{j=1}^d(\omega(X,V_j)-\sum_{k=1}^{j-1}c_{jk}(X)v_k)\partial_{v_j},$$
where $(v_1,\dots,v_d)$ is the generic point in $\gg^*$ with the coordinates computed with respect to the 
dual basis $V_1^*,\dots,V_d^*$. 
Then the formulas \eqref{Xstar}--\eqref{Z} follow by specializing 
the above formula and using \eqref{can} and \eqref{commutZX}--\eqref{commutZY}. 
\end{proof}

\subsection{The algebra of invariant differential operators on $\gg_0^*$}

\begin{proposition}\label{conclusion}
The unital associative algebra $\Diff(\gg_0^*)$ is generated by 
the set of first-order linear partial differential operators 
$$\begin{aligned}
 \{\partial_{\zeta^*}, & \partial_\zeta+\sum_{j=1}^m\eta_j^*\partial_{\xi_j}-\sum_{j=1}^d\xi_j^*\partial_{\eta_j}\} \\
&\cup\{\partial_{\xi_k},\partial_{\eta_k},
\partial_{\xi_k^*}-\frac{1}{2}(\zeta\partial_{\eta_k}+\eta_k^*\partial_{\zeta^*}), 
\partial_{\eta_k^*}+\frac{1}{2}(\zeta\partial_{\xi_k}+\xi_k^*\partial_{\zeta^*}) 
\mid 
k=1,\dots,m\}.
\end{aligned}$$
\end{proposition}

\begin{proof}
The proof consists of several steps. 
\begin{enumerate}
\item Until the next-to-last step we fix a first-order linear partial differential operator with real smooth coefficients on~$\gg_0^*$:   
$$D=\sum_{j=1}^m a_j\partial_{\xi_j}+ \sum_{j=1}^m b_j\partial_{\eta_j} +c\partial_\zeta 
+c^*\partial_{\zeta^*}+ \sum_{j=1}^m b_j^*\partial_{\eta_j^*}+\sum_{j=1}^m a_j^*\partial_{\xi_j^*}$$
It follows by \eqref{Xstar}--\eqref{Ystar} and \eqref{Z} 
that the conditions 
$$[\gamma_\omega(X_k^*),D]=[\gamma_\omega(Y_k^*),D]=[\gamma_\omega(Z),D]=0$$ 
for $k=1,\dots,m$  
are equivalent to the fact that the coefficients of $D$ depend only on 
$\zeta,\eta_1^*,\dots,\eta_m^*,\xi_1^*,\dots,\xi_m^*$. 

\item It follows by \eqref{X} that 
$$\begin{aligned}
0=&[\gamma_\omega(X_k),D] \\
 =&[-\partial_{\xi_k^*}-\frac{\eta_k^*}{2}\partial_{\zeta^*}+\zeta\partial_{\eta_k},D] \\
 =&-\sum_{j=1}^m(\partial_{\xi_k^*}a_j)\partial_{\xi_j} -\sum_{j=1}^m (\partial_{\xi_k^*}b_j)\partial_{\eta_j} 
  -(\partial_{\xi_k^*}c)\partial_\zeta \\
  &-(\partial_{\xi_k^*}c^*)\partial_{\zeta^*} -\sum_{j=1}^m (\partial_{\xi_k^*}b_j^*)\partial_{\eta_j^*}
   -\sum_{j=1}^m (\partial_{\xi_k^*}a_j^*)\partial_{\xi_j^*} \\
  &+\frac{b_k^*}{2}\partial_{\zeta^*}-c\partial_{\eta_k}
\end{aligned} $$
and this implies that 
\begin{align}
\label{X1}
& a_j,c,b_j^*,a_j^*\text{ depend only on }\zeta,\eta_1^*,\dots,\eta_m^*\\
\label{X2}
& b_j\text{ is also independent on }\xi_k^*\text{ if }j\ne k, \\
\label{X3}
& \partial_{\xi_k^*}b_k=-c, \\
\label{X4}
& \partial_{\xi_k^*}c^*=\frac{b_k^*}{2}.
\end{align}

\item It follows by \eqref{Y} that 
$$\begin{aligned}
0=&[\gamma_\omega(Y_k),D] \\
 =&[-\partial_{\eta_k^*}+\frac{\xi_k^*}{2}\partial_{\zeta^*}+\zeta\partial_{\xi_k},D] \\
 =&-\sum_{j=1}^m(\partial_{\eta_k^*}a_j)\partial_{\xi_j} -\sum_{j=1}^m (\partial_{\eta_k^*}b_j)\partial_{\eta_j} 
  -(\partial_{\eta_k^*}c)\partial_\zeta \\
  &-(\partial_{\eta_k^*}c^*)\partial_{\zeta^*} -\sum_{j=1}^m (\partial_{\eta_k^*}b_j^*)\partial_{\eta_j^*}
   -\sum_{j=1}^m (\partial_{\eta_k^*}a_j^*)\partial_{\xi_j^*} \\
  &-\frac{a_k^*}{2}\partial_{\zeta^*}+c\partial_{\xi_k}
\end{aligned} $$
and this implies that 
\begin{align}
\label{Y1}
& c,b_j^*,a_j^*\text{ depend only on }\zeta \\
\label{Y2}
& b_k\text{ depends only on }\zeta,\xi_k^*\\
\label{Y3}
& a_j\text{ is also independent on }\eta_k^*\text{ if }j\ne k, \text{ so }a_j\text{ depemds only on }\zeta,\eta_j^*,\\
\label{Y4}
& \partial_{\eta_k^*}a_k=c, \\
\label{Y5}
& \partial_{\eta_k^*}c^*=-\frac{a_k^*}{2}.
\end{align}

\item It follows by \eqref{Zstar} that 
$$\begin{aligned}
0=&[\gamma_\omega(Z^*),D] \\
 =&[\partial_{\zeta}
-\frac{1}{2}(\sum_{k=1}^m\xi_k^*\partial_{\eta_k}-\sum_{k=1}^m\eta_k^*\partial_{\xi_k}),D] \\
 =&\sum_{j=1}^m(\partial_\zeta a_j)\partial_{\xi_j} +\sum_{j=1}^m (\partial_\zeta b_j)\partial_{\eta_j} 
  +(\partial_\zeta c)\partial_\zeta \\
  &+(\partial_\zeta c^*)\partial_{\zeta^*} +\sum_{j=1}^m (\partial_\zeta b_j^*)\partial_{\eta_j^*}
   +\sum_{j=1}^m (\partial_\zeta a_j^*)\partial_{\xi_j^*} \\
  &+\sum_{k=1}^m\frac{a_k^*}{2}\partial_{\eta_k}-\sum_{k=1}^m\frac{b_k^*}{2}\partial_{\xi_k}
\end{aligned} $$
and this implies that 
\begin{align}
\label{Zstar1}
& c,b_j^*,a_j^*\text{ are constant}, \\
\label{Zstar2}
& c^*\text{ depends only on }\xi_1^*,\dots,\xi_m^*,\eta_1^*,\dots,\eta_m^*,\\
\label{Zstar3}
& \partial_\zeta a_j=\frac{b_j^*}{2}, \\
\label{Zstar4}
& \partial_\zeta b_j=-\frac{a_j^*}{2}.
\end{align}
\item 
It follows by \eqref{Y2}, \eqref{X3}, \eqref{Zstar1}, and \eqref{Zstar4} that 
there exists $b_{j0}\in\RR$ for which 
$$b_j=-c\xi_j^*-\frac{a_j^*}{2}\zeta+b_{j0}.$$
It follows by \eqref{Y3}, \eqref{Y4}, \eqref{Zstar1}, and \eqref{Zstar3} that 
there exists $a_{j0}\in\RR$ for which 
$$a_j=c\eta_j^*+\frac{b_j^*}{2}\zeta+a_{j0}.$$
It follows by \eqref{Zstar2}, \eqref{X4}, \eqref{Y5}, and \eqref{Zstar1} that 
there exists $c^*_0\in\RR$ for which 
$$c^*=\sum_{j=1}^m\frac{b_j^*}{2}\xi_j^*-\sum_{j=1}^m\frac{a_j^*}{2}\eta_j^*+c^*_0.$$
Therefore 
\begin{align}
D
=&\sum_{j=1}^m (c\eta_j^*+\frac{b_j^*}{2}\zeta+a_{j0})\partial_{\xi_j}
+ \sum_{j=1}^m (-c\xi_j^*-\frac{a_j^*}{2}\zeta+b_{j0})\partial_{\eta_j} 
+c\partial_\zeta \nonumber\\
&+(\sum_{j=1}^m\frac{b_j^*}{2}\xi_j^*-\sum_{j=1}^m\frac{a_j^*}{2}\eta_j^*+c^*_0)\partial_{\zeta^*}
+ \sum_{j=1}^m b_j^*\partial_{\eta_j^*}+\sum_{j=1}^m a_j^*\partial_{\xi_j^*} \nonumber
\end{align}
where $a_{j0},b_{j0},c,a_j^*,b_j^*,c^*\in\RR$. 
This shows that $D$ is a linear combination of the first-order linear partial differential operators in the set 
indicated in the statement, 
and it then follows by Corollary~\ref{aff_cor2} that that set  
generates the unital associative algebra~$\Diff(\gg_0^*)$. 
\end{enumerate}
\end{proof}

\subsection{Convolution operators on the Heisenberg group}

We recall from \cite{BB09} that the mapping 
$$\pi\colon G\to \Bc(L^2(H)),\quad 
\pi(\phi,X)f=\ee^{\ie\phi}\lambda(X)f$$
is a unitary irreducible representation of the connected, simply connected, 3-step nilpotent Lie group $G$ 
and the corresponding coadjoint orbit is 
$$\Oc=\{1\}\times\hg\times\hg^*\simeq\{1\}\times(\hg)^{**}\times\hg^*\subseteq\Fc^*\times\hg^*\simeq\gg^*. $$
We will perform the canonical identification 
$\Oc\longleftrightarrow\hg\times\hg^*$, $(X,\xi,1)\mapsto(X,\xi)$, 
and then the Weyl-Pedersen calculus for the representation $\pi$ 
can be defined by using the predual $\gg_0=\hg^*\dotplus\hg\subset(\RR\1+\hg^*)\rtimes_{\de\lambda}\hg=\gg$. 
We thus obtain   
$$\Op\colon\Sc'(\hg\times\hg^*)\to\Lc(\Sc(\hg),\Sc'(\hg))$$
and for $a\in \Sc(\hg\times\hg^*)$ we have $\widehat{a}\in\Sc(\hg^*\times\hg)$ and 
$$\Op(a)f=\iint\limits_{\hg^*\times\hg}\widehat{a}(\xi,X)\pi(\xi,X)f\de(\xi,X)
=\int\limits_{\hg}\int\limits_{\hg^*}\widehat{a}(\xi,X)\ee^{\ie\langle\xi,\cdot+\frac{1}{2}[X,\cdot]\rangle}\lambda(X)f\de\xi\de X$$
for every $f\in L^2(\hg)$. 
This implies that if $a\in \Sc'(\hg^*)$ and we think of it as a symbol in $\Sc'(\hg\times\hg^*)$ 
(that is, we will define $\Op(a):=\Op(\1\otimes a)$ where $\1$ stands for the constant function equal to 1 on $\hg$, 
as in Corollary~\ref{symbols})
and if we denote by $\widehat{a}\in \Sc'(\hg)$ the Fourier transform of $a$, then 
$$(\Op(a)f)(Y)=\int\limits_{\hg}\widehat{a}(X) f((-X)\cdot Y)\de X$$
for every $X\in\hg$ and $f\in\Sc(\hg)$ (see also \cite[Ex. 2.15(3)]{BB10}). 

Now let $\Ac(\hg^*)$ be the unital associative algebra 
generated by 
the set of first-order linear partial differential operators 
$$
 \{\partial_{\zeta^*}\} \cup\{
\partial_{\xi_k^*}-\frac{1}{2}\eta_k^*\partial_{\zeta^*}, 
\partial_{\eta_k^*}+\frac{1}{2}\xi_k^*\partial_{\zeta^*}
\mid 
k=1,\dots,m\}$$
and define 
$$\Ci_b(\hg^*):=\{a\in\Ci(\hg^*)\mid (\forall D\in\Ac(\hg^*))\quad Da\in L^\infty(\hg^*)\}.$$
As before, this is a Fr\'echet space with respect to the natural topology defined by the seminorms 
$a\mapsto\Vert Da\Vert_{L^\infty(\hg^*)}$ for $D\in\Ac(\hg^*)$.

\begin{corollary}
Then the following assertions are equivalent for $a\in C^\infty(\hg^*)$:
\begin{enumerate}
\item  $a\in \Ci_b(\hg^*)$,
\item For every differential operator $D\in\Ac(\hg^*)$ the corresponding convolution operator satisfies $\Op(Da)\in\BB(L^2(\hg))$. 
\end{enumerate}
Moreover we thus obtain a continuous linear map 
$\Op\colon \Ci_b(\hg^*)\to\BB(L^2(\hg))$. 
\end{corollary}

\begin{proof}
It follows by Proposition~\ref{conclusion} that the algebra $\Ac(\hg^*)$ is precisely the restriction of $\Diff(\gg_0^*)$ to the space of functions on $\gg_0^*=\hg\dotplus\hg^*$ that depend only on the variable in~$\hg^*$. 
Then the conclusion follows by Theorem~\ref{main_introd}. 
\end{proof}

\subsection*{Acknowledgments}
We wish to thank Prof. Bent \O rsted for help with a reference.
This research has been partially supported by the 
Grant of the Romanian National Authority for Scientific Research, CNCS-UEFISCDI, project number PN-II-ID-PCE-2011-3-0131.

\end{document}